\documentclass[11pt]{amsart}

\usepackage{amscd,amssymb,amsopn,amsmath,amsthm,mathrsfs,graphics,amsfonts,enumerate,verbatim,calc}

\usepackage[OT2,OT1]{fontenc}
\newcommand\cyr{%
\renewcommand\rmdefault{wncyr}%
\renewcommand\sfdefault{wncyss}%
\renewcommand\encodingdefault{OT2}%
\normalfont
\selectfont}
\DeclareTextFontCommand{\textcyr}{\cyr}

\usepackage{amssymb,amsmath}

\DeclareFontFamily{OT1}{rsfs}{}
\DeclareFontShape{OT1}{rsfs}{n}{it}{<-> rsfs10}{}
\DeclareMathAlphabet{\mathscr}{OT1}{rsfs}{n}{it}

\newcommand{\m}{\mathfrak{m}}

\numberwithin{equation}{section}
\newtheorem{theorem}{Theorem}[section]
\newtheorem{lemma}[theorem]{Lemma}
\newtheorem{proposition}[theorem]{Proposition}

\newtheorem{Problem}{Problem}

\newtheorem{Maintheorem}{Main Theorem}

\theoremstyle{definition}
\newtheorem{definition}[theorem]{Definition}
\newtheorem{remark}[theorem]{Remark}
\theoremstyle{remark}

\newtheorem{example}[theorem]{Example}
\newtheorem{counterexample}[theorem]{Counterexample}
\newtheorem{acknowledgement}{Acknowledgement}

\newcommand{\Ass}{\operatorname{Ass}}

\renewcommand{\ker}{\operatorname{Ker}}
\newcommand{\grade}{\operatorname{grade}}

\newcommand{\Spec}{\operatorname{Spec}}

\newcommand{\Ht}{\operatorname{ht}}

\newcommand{\Ext}{\operatorname{Ext}}

\newcommand{\Supp}{\operatorname{Supp}}

\newcommand{\Hom}{\operatorname{Hom}}
\newcommand{\Att}{\operatorname{Att}}

\newcommand{\Char}{\operatorname{char}}

\newcommand{\depth}{\operatorname{depth}}

\newcommand{\Proj}{\operatorname{Proj}}

\newcommand{\Pic}{\operatorname{Pic}}

\newcommand{\Lef}{\operatorname{Lef}}

\newcommand{\bcu}{\bigcup\limits}

\newcommand{\fm}{\frak{m}}
\newcommand{\fp}{\frak{p}}

\newcommand{\fa}{\frak{a}}

\newcommand{\fn}{\frak{n}}

\begin{document}
\title[A Study of quasi-Gorenstein rings II: Deformation of quasi-Gorenstein property]
{A Study of quasi-Gorenstein rings II: Deformation of quasi-Gorenstein property}

\author[K.Shimomoto]{Kazuma Shimomoto}
\address{Department of Mathematics, College of Humanities and Sciences, Nihon University, Setagaya-ku, Tokyo 156-8550, Japan}
\email{shimomotokazuma@gmail.com}

\author[N.Taniguchi]{Naoki Taniguchi}
\address{Department of Mathematics, Purdue University, West Lafayette, IN 47907 U.S.A}
\email{nendo@purdue.edu}

\author[E.Tavanfar]{Ehsan Tavanfar}
\address{School of Mathematics, Institute for Research in Fundamental Sciences (IPM),  P.O. Box: 19395-5746,Tehran, Iran}
\email{tavanfar@ipm.ir}

\thanks{The research of the third author was supported by a grant from IPM}


%\subjclass{13}
%\subjclass[2000]{Primary 13-XX}
%\subjclass[2000]{Primary ; Secondary}
%\date{\today \, (\printtime)}
%\date{\today}

\begin{abstract}
In the present article, we investigate the following deformation problem. Let $(R,\fm)$ be a local (graded local) Noetherian ring with a (homogeneous) regular element $y \in \fm$ and assume that $R/yR$ is quasi-Gorenstein. Then is $R$ quasi-Gorenstein? We give positive answers to this problem under various assumptions, while we present a counter-example in general. We emphasize that absence of the Cohen-Macaulay condition requires delicate and subtle studies.
\end{abstract}

\maketitle

\tableofcontents

\section{Introduction}

In this article, we study the deformation problem of the quasi-Gorenstein property on local Noetherian rings and construct some examples of non-Cohen-Macaulay, quasi-Gorenstein and normal domains. Recall that a local ring $(R,\fm)$ is quasi-Gorenstein, if it has a canonical module $\omega_R$ such that $\omega_R \cong R$. For completeness,
we state the general deformation problem as follows:

\begin{Problem}
Let $(R,\fm)$ be a local (graded local) Noetherian ring and $M$ be a nonzero finitely generated $R$-module with a (homogeneous) $M$-regular element $y \in \fm$. Assume that $M/yM$ has $\mathbf{P}$. Then does $M$ possess $\mathbf{P}$? 
\end{Problem}

By specializing $\mathbf{P}$=quasi-Gorenstein, we prove the following result by constructing an explicit example using Macaulay2 (see Theorem \ref{qGor}):

\begin{Maintheorem}
There exists an example of a local Noetherian ring $(R,\fm)$, together with a regular element $y \in \fm$ such that the following property holds: $R/yR$ is quasi-Gorenstein and $R$ is not quasi-Gorenstein.
\end{Maintheorem}

We notice that if a local ring $(R,\fm)$ is Cohen-Macaulay admitting a canonical module $\omega_R$ satisfying $\omega_R \cong R$, then it is Gorenstein. Thus, the local ring $R$ that appears in Main Theorem 1 is not Cohen-Macaulay. In the absence of Cohen-Macaulay condition, various aspects have been studied around the deformation problem in a recent paper \cite{TT16}. Our second main result is to provide some conditions under which the quasi-Gorenstein condition is preserved under deformation (see Theorem \ref{Theorem1}).

\begin{Maintheorem}
Let $(R,\fm)$ be a local Noetherian ring with a regular element $y \in \fm$, such that $R/yR$ is quasi-Gorenstein. If one of the following conditions holds, then $R$ is also quasi-Gorenstein.
\begin{enumerate}
\item
$R$ is of equal-characteristic $p>0$ that is $F$-finite and the Frobenius action on the local cohomology $H_{\fm}^{\dim R-1}(R/yR)$ is injective.

\item
$R$ is essentially of finite type over $\mathbb{C}$ and $R/yR$ has Du Bois singularities.

\item
$\Ext^1_{\widehat{R}}(\omega_{\widehat{R}},\omega_{\widehat{R}})=0$ and $0:_{\Ext^2_{\widehat{R}}(\omega_{\widehat{R}},\omega_{\widehat{R}})}y=0$, where $\widehat{R}$ is the $\fm$-adic completion of $R$.

\item
Both $R/yR$ and all of the formal fibers of $R$ satisfy Serre's $S_3$.

\item
All of the formal fibers of $R$ are Gorenstein, $R$ is quasi-Gorenstein on $\Spec^\circ(R/yR)$ and $\depth(R)\ge 4$.

\item
All of the formal fibers of $R$ are Gorenstein, $R/yR$ is Gorenstein on its punctured spectrum and $\depth(R)\ge 4$.

\item
$R$ is an excellent normal domain of equal-characteristic zero  such that $R[\frac{1}{y}]$ is also quasi-Gorenstein.
\end{enumerate}
\end{Maintheorem}

While Main Theorem 2 is concerned about local rings, we establish the following result for the graded local rings using algebraic geometry, including Lefschetz condition and vanishing of sheaf cohomology (see Theorem \ref{GrGorenstein}).

\begin{Maintheorem}
Let $R=\bigoplus_{n \ge 0} R_n$ be a Noetherian standard graded ring such that $y \in R$ is a regular element which is homogeneous of positive degree, $R_0=k$ is a field of characteristic zero. Suppose that $R/yR$ is a quasi-Gorenstein graded ring such that $X:=\Proj (R)$ is an integral normal variety and $X_1:=\Proj(R/yR)$ is nonsingular. Then $R$ is a quasi-Gorenstein graded ring.
\end{Maintheorem}

At the time of writing, the following problem remains open, because the example given in Theorem \ref{qGor} is not normal.

\begin{Problem}
Suppose that $(R,\fm)$ be a local (or graded local) ring with a regular element $y \in \fm$ such that $R/yR$ is a quasi-Gorenstein normal local (or graded local) domain. Is $R$ quasi-Gorenstein?
\end{Problem}

In the final section, we construct three non-trivial examples of quasi-Gorenstein normal local domains of depth equal to $2$ that are not Cohen-Macaulay (the final one being with arbitrary admissible dimension at least $3$) in Example \ref{DefNormal}. It will be interesting to ask the reader if any of these examples admits a non quasi-Gorenstein deformation. In the light of the above theorem, it is noteworthy to point out that any homogeneous deformation of the (standard) quasi-Gorenstein ring of Example \ref{DefNormal}(1) is again quasi-Gorenstein, provided that the deformation is standard of equal-characteristic zero. Let us end with a remark on the ubiquity of quasi-Gorenstein rings. 
\begin{enumerate}
\item[]
(Algebraic side): One can easily construct a local ring $(R,\fm)$ with a regular element $y \in \fm$ such that $R/yR$ is quasi-Gorenstein but not Gorenstein, which deforms to a quasi-Gorenstein ring $R$. For instance, take $(S,\fn)$ to be any non-Gorenstein and quasi-Gorenstein local ring. Then the trivial extension $R:=S[[y]]$ provides such an example. More interestingly, let $R$ be a complete local domain of arbitrary characteristic. Then it is shown in \cite{Ta18} that $R$ is dominated by a module-finite extension domain over $R$ that is quasi-Gorenstein and complete intersection at codimension $\le 1$. 

\item[]
(Geometric side):
The class of quasi-Gorenstein rings appears in Du Bois singularities. Indeed, we learn from Main Theorem 2(2) together with the main result in \cite{KS16} that if $R/yR$ has quasi-Gorenstein Du Bois singularity, then $R$ enjoys the same properties. This type of result will be essential for moduli problems as explained in \cite{KS16}. We also recall from \cite{K99} that normal quasi-Gorenstein Du Bois singularities are log canonical. This was previously known as a conjecture of Koll\'ar. On the other hand, it is known from \cite{KK10} that log canonical singularities are Du Bois.
\end{enumerate}

\section{Notation and auxiliary lemmas}

Let $(R,\fm)$ be a local Noetherian ring with Krull dimension $d:=\dim R$ and let $M$ be a finitely generated module. We say that $M$ is a \textit{canonical module} for $R$, if there is an isomorphism $M \otimes_R \widehat{R} \cong H_{\fm}^d(R)^{\vee}$, where $\widehat{R}$ is the $\fm$-adic completion of $R$. In general, assume that $R$ is a Noetherian ring and $M$ is a finitely generated $R$-module. Then $M$ is a \textit{canonical module} for $R$, if $M_{\fp}$ is a canonical module for the local ring $R_{\fp}$ for all $\fp \in \Spec(R)$ with $\dim R_\fp+\dim(R/\fp)=\dim R$. We will write a canonical module as $\omega_R$ in what follows. A local Noetherian ring $(R,\fm)$ is \textit{quasi-Gorenstein}, if there is an isomorphism $H^d_{\fm}(R)^\vee \cong \widehat{R}$. Equivalently, $R$ is quasi-Gorenstein, if $R$ admits a canonical module such that $\omega_R \cong R$ (see \cite{A83}). Let $R$ be a Noetherian ring admitting a canonical module $\omega_R$. Then $R$ is \textit{(locally) quasi-Gorenstein}, if the localization $R_\fp$ for $\fp \in \Spec(R)$ is quasi-Gorenstein in the sense above, or equivalently, $\omega_R$ is a projective module of constant rank 1. Let $R=\bigoplus_{n \ge 0}R_n$ be a graded Noetherian ring such that $R_0=k$ is a field. Then $R$ is \textit{quasi-Gorenstein}, if $\omega_R \cong R(a)$ for some $a \in \mathbb{Z}$ as graded $R$-modules. For a local ring $(R,\fm)$, we write the punctured spectrum $\Spec^{\circ} (R):=\Spec (R) \setminus \{\fm\}$. Let $I$ be an ideal of a ring $R$. Then let $V(I)$ denote the set of all prime ideals of $R$ that contain $I$. We also use some basic facts on \textit{attached primes}. For an Artinian $R$-module $M$, we denote by $\Att_R(M)$ the set of attached primes of $M$ (see \cite{BS12} for a brief summary).

We start by proving the following two auxiliary lemmas. The first lemma is a restatement of \cite[Lemma]{CL94} and we reprove it only for the convenience of the reader.

\begin{lemma}
\label{Lyubeznik}
Suppose that $(R,\fm)$ is a local Noetherian ring with $\depth(R)\ge 2$. Let $\mathfrak{a}$ be an ideal of $R$ such that $\m$ is not associated to $\mathfrak{a}$, the ideal $\mathfrak{a}$ is not contained in any associated  prime of $R$ and $\mathfrak{a}R_{\fp}$ is principal for $\fp \in \Spec^{\circ}(R)$. Then $\fa$ defines an element of $\Pic\big(\Spec^{\circ}(R)\big)$. Moreover if the line bundle attached to $\fa$ is a trivial element of  $\Pic\big(\Spec^{\circ}(R)\big)$, then $\mathfrak{a}$ is a principal ideal.
\end{lemma}

\begin{proof}
For each $\mathfrak{p}\in \Spec^{\circ}(R)$, we have $\fa R_\mathfrak{p}=(s)$ for some $s \in R_{\fp}$ by assumption. We need to show that we can choose $s$ as a regular element. Since $\fa$ is not contained in any associated prime of $R$, we have $\mathfrak{a}\nsubseteq \bcu_{\mathfrak{p} \in\Ass(R)}\mathfrak{p}$ by Prime Avoidance Lemma. So the $\mathcal{O}_{\Spec^{\circ}(R)}$-module $\widetilde{\mathfrak{a}}$ is invertible on $\Spec^{\circ}(R)$, which defines an element
$$
[\tilde{\mathfrak{a}}] \in \Pic\big(\Spec^{\circ}(R)\big).
$$
There are two exact sequences: $0\rightarrow \mathfrak{a}/\Gamma_\m(\mathfrak{a})\rightarrow H^0(\Spec^{\circ}(R),\tilde{\mathfrak{a}})\rightarrow H^1_\m(\mathfrak{a})\rightarrow 0$ and $\Gamma_\m(R/\mathfrak{a})\rightarrow H^1_\m(\mathfrak{a})\rightarrow H^1_\m(R)$, where the first exact sequence is due to \cite[III, Exercise 2.3.(e)]{H83} and \cite[III, Exercise 3.3.(b)]{H83}. We have $H^1_\m(R)=0$, because of $\depth(R) \ge 2$. We also have $\Gamma_\m(R/\mathfrak{a})=0$, because $\m$ is not associated to $\mathfrak{a}$. Hence we get $H^0(\Spec^{\circ}(R),\tilde{\mathfrak{a}})=\mathfrak{a}$ ($\Gamma_\m(\mathfrak{a})\subseteq \Gamma_\m(R)=0$). Now suppose that $\tilde{\mathfrak{a}}$ is the trivial element in $\Pic(\Spec^{\circ}(R))$. Then we have $\tilde{\mathfrak{a}}=\mathcal{O}_{\Spec^\circ(R)}$ and hence
$$
\mathfrak{a}=H^0(\Spec^{\circ}(R),\tilde{\mathfrak{a}})\cong H^0(\Spec^{\circ}(R),\mathcal{O}_{\Spec^{\circ}(R)})=R,
$$
where the last equality follows from the exact sequence
$$
0\rightarrow R/\Gamma_\m(R)\rightarrow H^0(\Spec^{\circ}(R),\mathcal{O}_{\Spec^{\circ}(R)}) \rightarrow H^1_\m(R)\rightarrow 0.
$$
\end{proof}

\begin{definition}
Let $\widehat{R}$ be the $\fm$-adic completion of a local ring $(R,\fm)$. We say that $R$ is \textit{formally unmixed}, if $\dim (\widehat{R}/P)=\dim (\widehat{R})$ for all $P \in \Ass(\widehat{R})$.
\end{definition}

\begin{lemma}
\label{unmixed}
Let $(R,\fm)$ be local Noetherian ring and suppose that $y \in \fm$ is a regular element such that $R/yR$ is quasi-Gorenstein. Then $R$ is formally unmixed.
\end{lemma}

\begin{proof}
First of all, recall that a quasi-Gorenstein local ring is unmixed by \cite[(1.8), page 87]{A83}. By definition of formal unmixedness, we can assume that $R$ is complete and we proceed by induction on the Krull dimension $d:=\dim (R)$. If $d \le 3$, then $R/yR$ is a quasi-Gorenstein ring of dimension at most $2$, which implies that $R/yR$ and $R$ are Gorenstein rings, hence $R$ is an unmixed ring. So suppose that $d\ge 4$ and the statement has been proved for smaller values than $d$. Pick $\mathfrak{q}\in \Ass(R)$. Then we have $\dim(R/\mathfrak{q})\ge 2$, because if otherwise, $\depth(R)\le \dim (R/\mathfrak{q})\le 1$ by \cite[Proposition 1.2.13]{BH98}, violating $\depth(R)\ge 3$. Thus, $\dim(R/\mathfrak{q}+yR)\ge 1$ (note that $y\notin \mathfrak{q}$, as $y$ is a regular element).  So we can choose $\mathfrak{p}/yR\in V(\mathfrak{q}+yR/yR)\backslash\{\m/yR\} \subset \Spec(R/yR)$ such that $\dim(R/\mathfrak{p})=1$. Since $R_{\mathfrak{p}}/yR_{\mathfrak{p}}$ is quasi-Gorenstein, the inductive hypothesis implies that $R_\mathfrak{p}$ is formally unmixed. Hence we have that $R_{\fp}$ is unmixed and $\dim (R/\mathfrak{q})-1\ge \Ht(\mathfrak{p}/\mathfrak{q})=\dim(R_\mathfrak{p}/\mathfrak{q} R_\mathfrak{p})=\Ht(\mathfrak{p})$. On the other hand, since $R/yR$ is a complete and quasi-Gorenstein local ring, it is catenary and equi-dimensional. Therefore, we have $\Ht(\m/yR)=\Ht(\mathfrak{p}/yR)+1$ and $\dim (R/\mathfrak{q})=\dim (R)$, as required.  
\end{proof}

Let us recall that the quasi-Gorenstein property admits a nice variant of deformation in \cite[Theorem 2.9]{TT16}:

\begin{theorem}[Tavanfar-Tousi]
\label{TT}
Let $(R,\fm)$ be a local Noetherian ring with a regular element $y \in \fm$. If $R/y^nR$ is quasi-Gorenstein for infinitely many $n\in \mathbb{N}$, then $R$ is quasi-Gorenstein.
\end{theorem}

\section{Deformation of quasi-Gorensteinness}

The aim of this section is to present some cases where the quasi-Gorenstein property deforms. We recall the notion of surjective elements which is given in \cite{HMS14}.

\begin{definition}
Let $(R,\fm)$ be a local Noetherian ring. A regular element $y \in \fm$ is called a \textit{surjective element}, if the natural map of local cohomology modules $H^i_{\fm}(R/y^nR) \to H^i_{\fm}(R/yR)$, which is induced by the natural surjection $R/y^nR \to R/yR$, is surjective for all $n>0$ and $i \ge 0$.
\end{definition}

In the parts $(1)$ and $(2)$ of the following theorem, the surjective elements will play a role. In $(2)$, a precise understanding of Du Bois singularities is not necessary, as we only need to use some established facts that follow from the definition.

\begin{theorem}
\label{Theorem1}
Let $(R,\fm)$ be a local Noetherian ring with a regular element $y \in \fm$, such that $R/yR$ is quasi-Gorenstein. If one of the following conditions holds, then $R$ is also quasi-Gorenstein.
\begin{enumerate}
\item
$R$ is of equal-characteristic $p>0$ that is $F$-finite and the Frobenius action on the local cohomology $H_{\fm}^{\dim R-1}(R/yR)$ is injective.

\item
$R$ is essentially of finite type over $\mathbb{C}$ and $R/yR$ has Du Bois singularities.

\item
$\Ext^1_{\widehat{R}}(\omega_{\widehat{R}},\omega_{\widehat{R}})=0$ and $0:_{\Ext^2_{\widehat{R}}(\omega_{\widehat{R}},\omega_{\widehat{R}})}y=0$, where $\widehat{R}$ is the $\fm$-adic completion of $R$.

\item
Both $R/yR$ and all of the formal fibers of $R$ satisfy Serre's $S_3$.

\item
All of the formal fibers of $R$ are Gorenstein, $R$ is quasi-Gorenstein on $\Spec^\circ(R/yR)$ and $\depth(R)\ge 4$.

\item
All of the formal fibers of $R$ are Gorenstein, $R/yR$ is Gorenstein on its punctured spectrum\footnote{According to \cite[9.5.7 Exercise]{BS12}, that a local ring $(R,\fm)$ is generalized Cohen-Macaulay is equivalent to the condition that $R$ is Cohen-Macaulay over the punctured spectrum, provided that $R$ admits the dualizing complex. Moreover, recall that a quasi-Gorenstein Cohen-Macaulay ring is Gorenstein and vice versa.} and $\depth(R)\ge 4$.

\item
$R$ is an excellent normal domain of equal-characteristic zero  such that $R[\frac{1}{y}]$ is also quasi-Gorenstein.
\end{enumerate}
\end{theorem}

\begin{proof}
In each of the cases $(4)$, $(5)$ and $(6)$, we can suppose that $R$ is complete without loss of generality. More precisely, we apply the assumption on the formal fibers and \cite[Theorem 4.1]{AG85} is needed in addition for part $(5)$ and $(6)$.  By Lemma \ref{unmixed}, $R$ is unmixed and in view of \cite[(1.8), page 87]{A83}, $R$ is quasi-Gorenstein if and only if it has a cyclic canonical module.

We prove the assertions $(1)$ and $(2)$ simultaneously. Then we prove $y$ is a surjective element for all $n>0$ and $i \ge 0$. When $R/yR$ has Du Bois singularities, then it follows from \cite[Lemma 3.3]{MSS16} that $y \in \fm$ is a surjective element. So assume that $R$ satisfies the condition $(1)$. Without loss of generality, we may assume that $R$ is complete. In this case, the Matlis dual of the Frobenius action $H^{\dim R-1}_{\fm}(R/yR) \hookrightarrow H^{\dim R-1}_{\fm}(F_*(R/yR))$ yields a surjection $\phi:F_*(R/yR) \twoheadrightarrow R/yR$ in view of the assumption that $R/yR \cong \omega_{R/yR}$. Then there is an element $F_*a \in F_*(R/yR)$ such that $\phi(F_*a)=1 \in R/yR$. Define a surjective $R$-module map $\Phi:F_*(R/yR) \to R/yR$ by letting $\Phi(F_*t):=\phi(F_*(at))$. Then the map $\Phi$ splits the Frobenius $R/yR \to F_*(R/yR)$. Hence $R/yR$ is $F$-split. As $F$-pure (split) rings are $F$-anti-nilpotent by \cite[Theorem 1.1 and Theorem 2.3]{M14}, it follows that $H^i_{\fm}(R/y^nR) \to H^i_{\fm}(R/yR)$ is surjective by \cite[Proposition 3.5]{MQ16}. 

We have proved that $y$ is a surjective element in $(1)$ and $(2)$. It follows from \cite[Proposition 3.3]{MQ16} that the multiplication map $H^i_{\fm}(R) \xrightarrow{\cdot y} H^i_{\fm}(R)$ is surjective for all $i \ge 0$. Letting $d=\dim R$, the short exact sequence $0 \to R \xrightarrow{\cdot y} R \to R/yR \to 0$ induces a short exact sequence
$$
0 \to H^{d-1}_{\fm}(R/yR) \to H^d_{\fm}(R) \xrightarrow{\cdot y} H^d_{\fm}(R) \to 0.
$$
Taking the Matlis dual of this exact sequence, we obtain the exact sequence:
$$
0 \to \omega_{\widehat{R}}  \xrightarrow{\cdot y} \omega_{\widehat{R}} \to \omega_{\widehat{R}/yR} \to 0.
$$
Hence we have $\omega_{\widehat{R}/y\widehat{R}} \simeq \omega_{\widehat{R}}/y\omega_{\widehat{R}}$. Nakayama's lemma allows us to write $\omega_{\widehat{R}} \simeq \widehat{R}/J$ for some ideal $J \subset \widehat{R}$. By Lemma \ref{unmixed}, $R$ is formally unmixed. Then we conclude that $J=0$ in view of \cite[(1.8)]{A83}. Hence $\omega_{\widehat{R}} \simeq \widehat{R}$.

We prove $(3)$ and argue by induction on dimension $d$. We may assume that $R$ is complete, $d\ge 4$ and that the statement is true in the case $d < 4$. Let us prove that $\Hom_{R/yR}(\omega_R/y\omega_R,\omega_R/y\omega_R)\cong R/yR$.  
By dualizing the exact sequence $H^{d-1}_\fm(R/yR) \to H^d_\fm(R) \xrightarrow{\cdot y} H^d_\fm(R) \to 0$, we have an exact sequence:
\begin{equation}
\label{exact sequence}
0\rightarrow \omega_R/y\omega_R\overset{g}{\rightarrow} \omega_{R/yR}\rightarrow C\rightarrow 0.
\end{equation}            
Consider the commutative diagram:
\begin{equation}
\label{CommutativeDiagram}
\begin{CD}
R/yR@>\alpha>> \Hom_{R/yR}(\omega_R/y\omega_R,\omega_R/y\omega_R) \\
@V\cong V R/yR\text{\ is\ }S_2 V @V\Hom(\text{id},g)V\text{injective}V\\
\Hom_{R/yR}(\omega_{R/yR},\omega_{R/yR})@>\Hom(g,\text{id})>> \Hom_{R/yR}(\omega_R/y\omega_R,\omega_{R/yR})
\end{CD}
\end{equation}
where $\alpha$ is the natural map $\overline{r}\mapsto \{t \mapsto \overline{r}t\}$. Upon the localization at $\fp \in \Spec^{\circ}(R/yR)$, the exact sequence (\ref{exact sequence}) becomes
\begin{center}
$0\rightarrow \omega_{R_\mathfrak{p}}/y\omega_{R_\mathfrak{p}}\overset{g}{\rightarrow} \omega_{R_\mathfrak{p}/yR_\mathfrak{p}}\rightarrow C_{\mathfrak{p}}\rightarrow 0$,
\end{center}
where $C_{\mathfrak{p}}$ is the Matlis dual to $H^{\dim (R_{\mathfrak{p}})-1}_{\mathfrak{p}R_\mathfrak{p}} (R_\mathfrak{p})\big/(y/1) H^{\dim (R_\mathfrak{p})-1}_{\mathfrak{p} R_\mathfrak{p}}(R_\mathfrak{p})$ (see \cite[Remark 2.3.(b)]{TT16}). But by our inductive hypothesis, $R_\mathfrak{p}$ is quasi-Gorenstein for each $\mathfrak{p} \in \Spec^\circ(R/yR)$ and so \cite[Corollary 2.8]{TT16} implies that $H^{\dim (R_\mathfrak{p})-1}_{\mathfrak{p} R_\mathfrak{p}}(R_\mathfrak{p})\big/(y/1) H^{\dim (R_\mathfrak{p})-1}_{\mathfrak{p} R_\mathfrak{p}}(R_\mathfrak{p})=0$ for each $\mathfrak{p} \in \Spec^\circ(R/yR)$. It follows that $C$ is of finite length. In particular, $\Ext^i_{R/yR}(C,\omega_{R/yR})=0$ for $i=0,1$ in view of the fact that $\omega_{R/yR} \cong R/yR$ and \cite[Theorem 6.2.2]{BS12}.

By applying $\Hom_{R/yR}(-,\omega_{R/yR})$ to the exact sequence (\ref{exact sequence}), we find that $\Hom(g,\text{id})$ is an isomorphism. Therefore, the commutative diagram (\ref{CommutativeDiagram}) in conjunction with the injectivity of $\Hom(\text{id},g)$ implies that $\alpha$ is an isomorphism.

Since $\depth(R/yR)\ge 2$ and $\Hom_{R/yR}(\omega_R/y\omega_R,\omega_R/y\omega_R)\cong R/yR$, we get $\depth(\omega_R/y\omega_R)\ge 1$. Applying the hypothesis $\Ext^1_{R}(\omega_{R},\omega_{R})=0$ and $0:_{\Ext^2_{R}(\omega_{R},\omega_{R})}y=0$ to the exact sequence $0 \to \omega_R \xrightarrow{\cdot y} \omega_R \to \omega_R/y\omega_R \to 0$, we get $\Ext^2_R(\omega_R/y\omega_R,\omega_R)=0$. So it follows from \cite[Lemma 3.1.16]{BH98} that $\Ext^1_{R/yR}(\omega_R/y\omega_R,\omega_R/y\omega_R)=0$. Set $N:=\omega_R/y\omega_R$ and assume that $z\in R/yR$ is an $N$-regular element. This choice is possible due to $\depth(\omega_R/y\omega_R)\ge 1$. By applying $\Hom_{R/yR}(N,-)$ to the exact sequence $0 \to N \xrightarrow{\cdot z} N \to N/zN \to 0$, we get an exact sequence:
$$
0 \to \Hom_{R/yR}(N,N)/z\Hom_{R/yR}(N,N) \to \Hom_{R/yR}(N,N/zN) \to \Ext^1_{R/yR}(N,N),
$$
which gives
$$
\Hom_{R/yR}(N,N)/z\Hom_{R/yR}(N,N)\cong \Hom_{R/yR}(N,N/zN).
$$
So we have $\depth(N/zN) \ge 1$, because if otherwise, we would have $\depth(\Hom_{R/yR}(N,N)) \le 1$, which contradicts $\Hom_{R/yR}(N,N)\cong R/yR$ and $\depth(R/yR) \ge 2$ as proved above. It follows that $\depth(\omega_R/y\omega_R)\ge 2$. Thus, we have $\depth(\omega_R)\ge 3$ and $\m \notin \text{Att}(H^{d-1}_{\m}(R))$ in view of \cite[Lemma 2.1 (2)(i)]{AG85}. We claim that 
$$
y\notin \bcu_{\mathfrak{p}\in \text{Att}_{R}\big(H^{d-1}_{\m}(R)\big)}\mathfrak{p}.
$$
Indeed, this implies that the multiplication map $H^{d-1}_{\m}(R) \xrightarrow{\cdot y} H^{d-1}_{\m}(R)$ is surjective in view of \cite[Proposition 7.2.11]{BS12}. So suppose to the contrary that $y\in \mathfrak{p}$ for some $\mathfrak{p}\in \text{Att}_R\big(H^{d-1}_{\m}(R)\big)$. Then by Shifted Localization Theorem, we have $y/1\in \mathfrak{p} R_\mathfrak{p}\in \text{Att}_{R_\mathfrak{p}}\big(H^{\Ht(\mathfrak{p})-1}_{\mathfrak{p} R_\mathfrak{p}}(R_\mathfrak{p})\big)$. As we already proved that $\mathfrak{p}\neq \m$, the induction hypothesis implies that $R_{\mathfrak{p}}$ is quasi-Gorenstein  and by \cite[Corollary 2.8]{TT16}, we must get
$$
y/1\notin \bcu_{\mathfrak{q} R_\mathfrak{p}\in \text{Att}_{R_{\mathfrak{p}}}\big(H^{\Ht(\mathfrak{p})-1}_{\mathfrak{p} R_\mathfrak{p}}(R_\mathfrak{p})\big)}\mathfrak{q} R_\mathfrak{p},
$$
a contradiction. By a similar argument as in part $(1)$ or $(2)$, we can establish $\omega_R\cong R$.

We prove $(4)$. This can be reduced to the situation of part $(5)$, using the Noetherian induction. However, we will deduce it via a simpler proof than the proof of part $(5)$. Since both $R/yR$ and the formal fibers of $R/yR$ have $S_3$, the $\fm$-adic completion of $R/yR$ satisfies the same hypothesis. So let us assume that $R$ is complete. Now $R/yR$ is a quasi-Gorenstein complete local with $S_3$, so we have $H^{d-2}_\m(R/yR)=0$ in view of \cite[Corollary 1.15]{S98}. It follows that the multiplication map $H^{d-1}_\m(R) \xrightarrow{\cdot y} H^{d-1}_\m (R)$ is injective on the $\m$-torsion module $H^{d-1}_\m (R)$, which yields $H^{d-1}_\m(R)=0$. We conclude that $ R/yR\cong \omega_{R/yR}\cong \omega_R/y\omega_R$, showing that $\omega_R$ is cyclic, as required.

We prove $(5)$. Notice that by \cite[Corollary 2.8]{TT16} together with Theorem \ref{TT}, we easily deduce that $R$ is quasi-Gorenstein on $\Spec^\circ(R/yR)$ if and only if $R/y^nR$ is quasi-Gorenstein on $\Spec^{\circ}(R/y^nR)$ for each $n \ge 2$. Suppose that $R$ satisfies these equivalent conditions and $\depth(R)\ge 4$. Moreover, since $R$ has Gorenstein formal fibers, we can suppose that $R$ is a complete local ring without loss of generality. Then both $\omega_R/y^n\omega_R$ and $\omega_{R/y^nR}$ define line bundles on $\Spec^{\circ}(R/y^nR)$.  We claim that these line bundles are identical on $\Spec^{\circ}(R/y^nR)$. By \cite[Remark 2.3]{TT16}, there exists a natural embedding: $\omega_R/y^n\omega_R \hookrightarrow \omega_{R/y^nR}$ whose cokernel $C$ is locally (by Matlis duality) dual to $H^{\dim(R_\mathfrak{p})-1}_{\mathfrak{p} R_\mathfrak{p}}(R_\mathfrak{p})/y^nH^{\dim(R_\mathfrak{p})-1}_{\mathfrak{p} R_\mathfrak{p}}(R_\mathfrak{p})$ for each $\mathfrak{p} \in \Spec(R/yR)$. Since both $R_\mathfrak{p}$ and $R_\mathfrak{p}/y^nR_\mathfrak{p}$ are quasi-Gorenstein for each $\mathfrak{p}\in \Spec^{\circ}(R/yR)$, we have $C_{\fp}=0$ for $\fp \in \Spec^{\circ}(R/y^nR)$ in view of \cite[Corollary 2.8]{TT16} and hence our claim follows. There is a group homomorphism:
$$
\pi_n:\Pic\big(\Spec^{\circ}(R/y^nR)\big) \to \Pic\big(\Spec^{\circ}(R/y^{n-1}R)\big),
$$
which is induced by the natural surjection $M \mapsto M/y^{n-1}M$ for each $n\ge 2$. Since $R/yR$ is quasi-Gorenstein, we have
$$
0=[\omega_{R/yR}]=[\omega_R/y\omega_R]=[\pi_2(\omega_{R}/y^2\omega_{R})]=[\pi_2(\omega_{R/y^2R})],
$$
that is to say, we have $[\omega_{R/y^2R}]\in \ker(\pi_2)$. Since $\depth(R)\ge 4$, arguing as in \cite[III, Exercise 4.6]{H83},  we can apply \cite[III, Exercise 2.3(e)]{H83}, \cite[III, Exercise 3.3(b)]{H83} and \cite[III, Theorem 3.7]{H83} to see that $\pi_2$ is injective and thus, $[\omega_{R/y^2R}]$ is trivial in $\text{Pic}\big(\Spec^{\circ}(R/y^2R)\big)$. By considering the maps $\pi_n$ inductively and using a different but similar exact sequence as in \cite[III, Exercise 4.6]{H83},\footnote{More precisely, consider the exact sequence $0 \to \mathcal{O}_{1} \xrightarrow{g}\mathcal{O}^*_{n+1} \to \mathcal{O}^*_n \to 0$, where $\mathcal{O}^*_n$ denotes the sheaf of the group of invertible elements on $\Spec^{\circ}(R/y^nR)$ and $g$ is defined by $t \mapsto 1+ty^n$.} we can deduce that $[\omega_{R/y^nR}]=0$ as an element of $\Pic\big(\Spec^{\circ}(R/y^nR)\big)$ for each $n\ge 1$.

Suppose to the contrary that $R$ is not quasi-Gorenstein. Then according to Theorem \ref{TT}, there exists an integer $n\ge 2$ such that $R/y^nR$ is not quasi-Gorenstein. For each $n\ge 2$, $R/y^nR$ satisfies Serre's $S_2$-condition, we have $H^{d-1}_\m(\omega_{R/y^nR})\cong E_{R/y^nR}(R/\m)$ in view of \cite[Remark 1.4]{AG85}, because it is quasi-Gorenstein on $\Spec^{\circ}(R/y^nR)$ and $\depth(R/y^nR)\ge 3$ by assumption. Since $R/y^nR$ is generically Gorenstein, $\omega_{R/y^nR} \cong \fa$ for an ideal $\mathfrak{a} \subseteq R/y^nR$ by applying \cite[Lemma 1.4.4]{BH98} and \cite[1.4.18]{BH98}. Since $R/y^nR$ has $S_2$, but is not quasi-Gorenstein, after applying the functor $\Gamma_\m(-)$ to the exact sequence $0\rightarrow \mathfrak{a}\rightarrow R/y^nR\rightarrow (R/y^nR)/\mathfrak{a}\rightarrow 0$, we conclude that $\Ht(\mathfrak{a})\le 1$; otherwise we would get $H^{d-1}_\m(R/y^nR)\cong H^{d-1}_\m(\omega_{R/y^nR})\cong E_{R/y^nR}(R/\m)$, contradicting to our hypothesis that $R/y^nR$ is not quasi-Gorenstein. On the other hand, $\mathfrak{a}$ has trivial annihilator, because $R/y^nR$ is unmixed by \cite[(1.8), page 87]{A83} and \cite[Lemma 1.1]{AG85}. So it follows that $\Ht(\mathfrak{a})=1$. 
Since $\mathfrak{a}$ satisfies $S_2$, we get $\Gamma_\m\big((R/y^nR)/\mathfrak{a}\big)\cong H^1_\m(\mathfrak{a})=0$. Therefore,  $\mathfrak{a}$ satisfies the hypothesis of Lemma \ref{Lyubeznik} and hence it is principal, i.e. $R/y^nR$ is quasi-Gorenstein. But this is a contradiction and we must get that $R/y^nR$ is quasi-Gorenstein for all $n>0$. That is, $R$ is quasi-Gorenstein.

The assertion $(6)$ is a special case of part $(5)$.

Finally, we prove the assertion $(7)$. Suppose the contrary. Then using the Noetherian induction, we may assume that $R_{\fp}$ is quasi-Gorenstein for all $ \fp \in \Spec^\circ(R/yR)$. Since $R[\frac{1}{y}]$ is quasi-Gorenstein by assumption, $\omega_R$ defines an element of $\Pic\big(\Spec^{\circ}(R)\big)$ which, in view of our hypothesis, belongs to
$$
\ker \Big(\Pic\big(\Spec^{\circ}(R)\big) \to \Pic\big(\Spec^{\circ}(R/yR)\big)\Big).
$$ 
Then by virtue of a theorem of Bhatt and de Jong \cite[Theorem 0.1]{Bd14}, $\omega_R$ is the trivial element in $\Pic(\Spec^{\circ}(R))$. Then the desired conclusion follows by applying Lemma \ref{Lyubeznik} to $R$.
\end{proof}

Let us prove a positive result in the graded normal case. First, we prepare a few lemmas.

\begin{lemma}
\label{Helpful}
Suppose that $R=\bigoplus_{n \ge 0} R_n$ is a Noetherian standard graded ring with $\fm:=\bigoplus_{n>0} R_n$ and that $M$ is a finitely generated graded $R$-module with $\grade_{\fm}(M)\ge 2$. Then
$$
M \cong \bigoplus_{n\in \mathbb{Z}}H^0\big(X,\widetilde{M}(n)\big),
$$
where we put $X:=\Proj(R)$.
\end{lemma}

\begin{proof}
According to \cite[(2.1.5)]{Gro1}, there is an exact sequence
$$
0 \to H_{\fm}^0(M) \to M \to \bigoplus_{n\in\mathbb{Z}}H^0\big(X,\widetilde{M}(n)\big) \to H_{\fm}^1(M) \to 0
$$
under the stated hypothesis on $(R,\fm)$. Since $\grade_{\fm}(M)\ge 2$ by assumption, we have the claimed isomorphism.
\end{proof}

We need some tools from algebraic geometry.

\begin{definition}[Lefschetz condition]
Let $X$ be a Noetherian scheme and let $Y \subset X$ be a closed subscheme. Denote by $\widehat{(~)}$ the formal completion along $Y$. Then we say that the pair $(X,Y)$ satisfies the \textit{Lefschetz condition}, written as $\Lef(X,Y)$, if for every open neighborhood $U$ of $Y$ in $X$ and a locally free sheaf $\mathcal{F}$ on $U$, there exists an open subset $U'$ of $X$ with $Y \subset U' \subset U$ such that the natural map
$$
H^0(U',\mathcal{F}|_{U'}) \to H^0(\widehat{X},\widehat{\mathcal{F}})
$$
is an isomorphism.
\end{definition}

The Lefschetz condition has been used to study the behavior of Picard groups or algebraic fundamental groups under the restriction maps. We refer the reader to 
\cite[Chapter IV]{H70} for these topics.

\begin{lemma}
\label{Lefschetz}
Let $X$ be an integral projective variety of dimension $\ge 2$ over a field of characteristic zero and let $D \subset X$ be a nonsingular effective ample divisor. Then the pair $(X,D)$ satisfies the Lefschetz property $\Lef(X,D)$.
\end{lemma}

\begin{proof}
Since $D$ is locally principal and nonsingular, there exists an open neighborhood $D \subset V$ in $X$ such that $V$ is nonsingular and dense in $X$. By Hironaka's theorem of desingularization, there exists a nonsingular integral variety $Y$ and a proper birational morphism $\pi:Y \to X$ such that $\pi^{-1}(V) \cong V$. By \cite[Lemma 3.4]{RS06},\footnote{To apply the lemma, we need that $X \setminus D$ is affine and the cohomological dimension of $Y \setminus \pi^{-1}(D)$ is at most $\dim Y-1$; these are satisfied in our case in view of \cite[Corollary 3.5 at page 98]{H70}.} there exists an effective Cartier divisor $E \subset Y$ such that either $E=0$ or $\dim \pi(\Supp(E))=0$ and
\begin{equation}
\label{factorization}
H^0\big(Y,\mathcal{F} \otimes \mathcal{O}_Y(E)\big) \cong H^0(\widehat{Y},\widehat{\mathcal{F}})
\end{equation}
for a fixed coherent reflexive sheaf $\mathcal{F}$ on $Y$, where $\mathcal{F}$ is locally free around some neighborhood of $D \cong \pi^{-1}(D)$. Here, $\widehat{(~)}$ is the completion along $\pi^{-1}(D) \subset Y$. For any open neighborhood $\pi^{-1}(D) \subset U$ such that $U \cap \Supp(E)=\emptyset$, the map $(\ref{factorization})$ factors as
$$
H^0\big(Y,\mathcal{F} \otimes \mathcal{O}_Y(E)\big) \to H^0\big(U,\mathcal{F} \otimes \mathcal{O}_Y(E)\big) \to H^0(\widehat{Y},\widehat{\mathcal{F}})
$$
and we have an isomorphism $H^0(U,\mathcal{F} \otimes \mathcal{O}_Y(E)) \cong H^0(U,\mathcal{F})$. Therefore,
\begin{equation}
\label{factorization2}
H^0(U,\mathcal{F}) \to H^0(\widehat{Y},\widehat{\mathcal{F}})
\end{equation}
is surjective. Let us prove that $(\ref{factorization2})$ is injective. Let $\mathcal{I}$ be the ideal sheaf of $D':=\pi^{-1}(D) $ (as a closed subscheme of $U$). Then we have a short exact sequence: $0 \to \mathcal{I}^n \to \mathcal{O}_U \to \mathcal{O}_{D'_n} \to 0$, where $D_n$ is the $n$-th infinitesimal thickening of $D'$. Now we get a short exact sequence 
$$
0 \to \mathcal{I}^n\mathcal{F} \to \mathcal{F} \to \mathcal{F}/\mathcal{I}^n\mathcal{F} \to 0.
$$
Taking cohomology, we get an exact sequence $0 \to H^0(U,\mathcal{I}^n\mathcal{F}) \to H^0(U,\mathcal{F}) \to H^0(D'_n,\mathcal{F}/\mathcal{I}^n\mathcal{F})$. Using \cite[Chapter II, Proposition 9.2]{H83},\footnote{There is a result asserting that the cohomology functor commutes with inverse limit functor under the Mittag-Leffler condition; see \cite[Proposition 8.2.5.3]{Illusie}.} one gets an exact sequence
$$
0 \to \varprojlim_n H^0(U,\mathcal{I}^n\mathcal{F}) \to H^0(U,\mathcal{F}) \to H^0(\widehat{Y},\widehat{\mathcal{F}}),
$$
where the latter map coincides with $(\ref{factorization2})$. So it suffices to prove that $\varprojlim_n H^0(U,\mathcal{I}^n\mathcal{F})=0$. In view of \cite[Chapter II, Proposition 9.2]{H83}, one is reduced to proving that $\varprojlim_n \mathcal{I}^n\mathcal{F}=0$. Since this question is local, we may assume that $U=\Spec (R)$ for a Noetherian ring $R$. Since $Y$ is an integral variety, its open subset $U$ is also integral. Therefore, $R$ is a Noetherian domain. We have
$$
\widetilde{I^n F} \cong \mathcal{I}^n\mathcal{F}
$$
for an ideal $I \subset R$ and a projective $R$-module $F$ of finite rank. However, $R$ is a Noetherian domain, it follows from Krull's intersection theorem that $\bigcap_{n>0} I^nF=0$ and thus
$$
\varprojlim_n \mathcal{I}^n\mathcal{F}=0,
$$
as desired.

For any locally free sheaf $\mathcal{G}$ over an open subset $W \subset X$ such that $D \subset W \subset V$ with $V$ as in the beginning of the proof, since $\pi^{-1}(W) \cong W$, we have the commutative diagram:
$$
\begin{CD}
H^0\big(\pi^{-1}(W),\pi^*\mathcal{G}|_{\pi^{-1}(W)}\big) @>\simeq>> H^0\big(\widehat{Y},\widehat{\pi^*\mathcal{G}|_{\pi^{-1}(W)}}\big) \\
@| @AAA \\
H^0(W,\mathcal{G}) @>>> H^0(\widehat{X},\widehat{\mathcal{G}}) \\
\end{CD} 
$$
where the vertical map on the right is induced by the map $\pi$ and the horizontal map on the top is an isomorphism, due to $(\ref{factorization2})$. On the other hand, letting $\mathcal{J}$ be the ideal sheaf of $D \subset W$, we have isomorphisms $\pi^{-1}(D)_n \cong D_n$ and
$$
H^0(\widehat{X},\widehat{\mathcal{G}}) \cong \varprojlim_{n>0} H^0(D_n,\mathcal{G}/\mathcal{J}^n\mathcal{G}) \cong
\varprojlim_{n>0} H^0\big(\pi^{-1}(D)_n,\pi^*\mathcal{G}/\pi^{-1}(\mathcal{J})^n\pi^*\mathcal{G}\big) \cong H^0\big(\widehat{Y},\widehat{\pi^*\mathcal{G}|_{\pi^{-1}(W)}}\big).
$$
In summary, $H^0(W,\mathcal{G}) \to H^0(\widehat{X},\widehat{\mathcal{G}})$ is an isomorphism, which shows that the pair $(X,D)$ satisfies $\Lef(X,D)$, as desired.

%Then applying \cite[Chapter IV, Corollary 1.2]{H70}, the pair $(Y,\pi^{-1}(D))$ satisfies the condition $\Lef(Y,\pi^{-1}(D))$. Let $i_n:D_n \hookrightarrow X$ be the $n$-th infinitesimal thickening of $D$. Then $i_n$ factors as $D_n \hookrightarrow U \hookrightarrow X$, which fits into the commutative diagram:
%$$
%\begin{CD}
%\pi^{-1}(D)_n=\pi^{-1}(D_n) @>>> \pi^{-1}(U) @>>> Y \\
%@| @| @V\pi VV \\
%D_n @>>> U @>>> X
%\end{CD}
%$$
%Using this diagram, we get $
%\widehat{\mathcal{F}} \cong \varprojlim i_n^*\mathcal{F} \cong \varprojlim j_n^*(\pi^*\mathcal{F}) \cong \widehat{\pi^*\mathcal{F}}$ regarded as sheaves of abelian groups on the topological space $D$, where $j_n$ is the composite of the top horizontal maps. Then we have
%\begin{equation}
%\label{iso1}
%H^0(\widehat{X},\widehat{\mathcal{F}}) \cong H^0(\widehat{Y},\widehat{\pi^*\mathcal{F}}).
%\end{equation}
%On the other hand, for any open neighborhood $D \subset U' \subset U$ with $U$ as above, we have an isomorphism:
%\begin{equation}
%\label{iso2}
%H^0(U',\mathcal{F}|_{U'}) \cong H^0(\pi^{-1}(U'),\pi^*\mathcal{F}|_{\pi^{-1}(U')}).
%\end{equation}
%Combining $(\ref{iso1})$ and $(\ref{iso2})$ together with $\Lef(Y,\pi^{-1}(D))$,
%we find that the pair $(X,D)$ also satisfies $\Lef(X,D)$.

\end{proof}

Let us prove the following result.

\begin{theorem}
\label{GrGorenstein}
Let $R=\bigoplus_{n \ge 0} R_n$ be a Noetherian standard graded ring such that $y \in R$ is a regular element which is homogeneous of positive degree, $R_0=k$ is a field of characteristic zero. Suppose that $R/yR$ is a quasi-Gorenstein graded ring such that $X:=\Proj (R)$ is an integral normal variety and $X_1:=\Proj(R/yR)$ is nonsingular. Then $R$ is a quasi-Gorenstein graded ring.
\end{theorem}

\begin{proof}
Let us fix notation: $R_{(n)}:=R/y^nR$, $\fm:=\bigoplus_{n \ge 1} R_n$ and $X_n:=\Proj (R/y^nR)$ for each $n >0$. Since $R_{(n)}$ is a standard graded ring over the field $k$, the sheaves $\mathcal{O}_{X_n}(m)$ are invertible for $m \in \mathbb{Z}$ and $n > 0$.\footnote{The paper \cite{GW78} considers a more generalized version of standard graded rings, known as "condition $(\#)$" in \cite[page 206]{GW78}.} Assume that $R/yR$ is quasi-Gorenstein. Then:
\begin{equation}
\label{divisorial}
\depth R \ge 3,~\mathcal{O}_{X}(n)~\mbox{is invertible and}~\widetilde{\omega_R}(n)~\mbox{is an}~S_2\mbox{-sheaf}.
\end{equation}

Now let us prove that $R$ is quasi-Gorenstein. First, assume that $\dim X \le 2$, or equivalently $\dim R \le 3$. Since $R/yR$ is quasi-Gorenstein, it has $\dim R/yR=\depth R/yR \ge 2$, in which case it is immediate to see that $R$ is a Gorenstein graded ring. In what follows, let us assume that $\dim X \ge 3$ and set $d:=\deg(y)$. Then we have a short exact sequence: $0 \to y^{n}R/y^{n+1}R \to R/y^{n+1}R \to R/y^nR \to 0$. Put $\mathcal{O}_{X_1}(-dn):=\widetilde{R/yR}(-dn)$. Then there is an isomorphism
$$
\mathcal{O}_{X_1}(-dn) \xrightarrow{\cdot y^n} \big(\widetilde{y^{n}R/y^{n+1}R}\big)~\mbox{as}~\mathcal{O}_{X_1}\mbox{-modules}.
$$
Then we get an exact sequence of abelian sheaves:
\begin{equation}
\label{shortexact}
0\rightarrow \mathcal{O}_{X_1}(-dn) \xrightarrow{\alpha} \mathcal{O}^*_{X_{n+1}}\rightarrow \mathcal{O}_{X_n}^*\rightarrow 0
\end{equation}
on the topological space $X_1$, where $\alpha(t):=1+ty^n$. Since $\mathcal{O}_{X_1}(-dn)$ is the dual of an ample divisor for $n > 0$, we have $H^1\big(X_1,\mathcal{O}_{X_1}(-dn)\big)=0$ for $n>0$ by Kodaira's vanishing theorem. Hence the map between Picard groups induced by $(\ref{shortexact})$
\begin{equation}
\label{KernelPic}
\pi_{n+1}:\Pic(X_{n+1}) \to \Pic(X_{n})
\end{equation}
is injective in view of \cite[III, Exercise 4.6]{H83}. Denote by $a:=a(R_{(1)})$ the $a$-invariant of $R_{(1)}$. Then we have $\omega_{R_{(1)}} \cong R_{(1)}(a)$ and thus by \cite[Lemma (5.1.2)]{GW78}, 
$$
\widetilde{\omega_{R_{(1)}}}(-a) \cong \widetilde{\omega_{R_{(1)}}} \otimes \mathcal{O}_{X_1}(-a) \cong \mathcal{O}_{X_1}(a)\otimes\mathcal{O}_{X_1}(-a) \cong \mathcal{O}_{X_1}.
$$
Since $y \in R$ is regular and $X_1 \subset X$ is a nonsingular divisor, $X$ is nonsingular in a neighborhood of $X_1$ and $X_1=X_2=\cdots$ as topological spaces. In particular, $X_n$ is a Gorenstein scheme for $n\ge 1$. By \cite[Theorem (A.3.9)]{TW89}, we have $\big[\widetilde{\omega_{R_{(n)}}}\big] \in \Pic(X_n)$  for $n \ge 1$. Consider the short exact sequence $0 \to R(-dn) \xrightarrow{\cdot y^n} R \to R_{(n)} \to 0$. By \cite[Proposition (2.2.9)]{GW78}, we get an injection:
$$\
\big(\omega_R/y^n\omega_R\big)(dn) \hookrightarrow \omega_{R_{(n)}}.
$$
Then an inspection of the proof of \cite[Proposition (2.2.10)]{GW78}, together with the fact that $X_n$ is Gorenstein, yields that
$$
\widetilde{\big(\omega_R/y^n\omega_R\big)}(dn) \cong \widetilde{\omega_{R_{(n)}}}~\mbox{for}~n>0.
$$ 
Hence we have $\big[\widetilde{\big(\omega_R/y^n\omega_R\big)}(dn) \big] \in \Pic(X_n)$ and $\big[\widetilde{\big(\omega_R/y^n\omega_R\big)}(m)\big] \in \Pic(X_n)$ for $m \in \mathbb{Z}$ and $n\ge 2$. Since $\big[\widetilde{\big(\omega_R/y\omega_R\big)}(d-a)\big] \in \Pic(X_1)$ is trivial, it follows from $(\ref{KernelPic})$ that 
$$
\widetilde{\big(\omega_R/y^{n+1}\omega_R\big)}(2d-a) \cong \mathcal{O}_{X_{n+1}}(d)
$$
for $n>0$. Since $X_1 \subset X$ is a nonsingular divisor, there is an open neighborhood $X_1 \subset U$ such that $U$ is nonsingular. In particular, it follows that $\widetilde{\omega_R}(2d-a)\big|_{U}$ is a line bundle. There are isomorphisms for all $n>0$ and $m \in \mathbb{Z}$:
$$
\mathcal{O}_{X_{n+1}}(d+m) \cong \widetilde{\big(\omega_R/y^{n+1}\omega_R\big)}(2d-a+m) \cong \widetilde{\omega_R}(2d-a+m)\Big/\widetilde{y^{n+1}\omega_R}(2d-a+m).
$$
Hence we get $\widehat{\mathcal{O}_X(d+m)} \cong \widehat{\widetilde{\omega_R}(2d-a+m)}$, where $\widehat{(~)}$ is the formal completion along the closed subscheme $X_1 \subset X$. Therefore,
$$
\big[\widetilde{\omega_R}(2d-a+m)\big|_U\big]-\big[\mathcal{O}_X(d+m)\big|_U\big] \in \ker\big(\Pic(U) \rightarrow \Pic(\widehat{X})\big).
$$
Notice that $X_1 \subset X$ is a nonsingular Cartier divisor and the pair $(X,X_1)$ satisfies the property $\Lef(X,X_1)$ in view of Lemma \ref{Lefschetz}. So after possibly shrinking $U$ more, it follows that
\begin{equation}
\label{CoherentSheavesIdentity}
H^0\big(U,\widetilde{\omega_R}(2d-a+m)\big) \cong H^0\big(\widehat{X},\widehat{\widetilde{\omega_R}(2d-a+m)}\big) \cong H^0\big(\widehat{X},\widehat{\mathcal{O}_{X}(d+m)}\big) \cong H^0\big(U,\mathcal{O}_{X}(d+m)\big).
\end{equation}
We claim that $Z:=X \setminus U$ is zero-dimensional. Indeed, the complement $X \setminus X_1$ is affine. On the other hand, $Z$ is a proper scheme over $k$  that is contained in $X \setminus X_1$, so $Z$ must be a zero-dimensional closed set in $X$. Using these facts together with the hypothesis $\dim X \ge 3$ and $(\ref{divisorial})$, we have an exact sequence: 
$$
0=H_Z^0\big(X,\widetilde{\omega_R}(m)\big) \to H^0\big(X,\widetilde{\omega_R}(m)\big) \to H^0\big(U,\widetilde{\omega_R}(m)\big) \to H_Z^1\big(X,\widetilde{\omega_R}(m)\big)=0
$$
in view of \cite[III, Exercise 2.3 (e) and (f)]{H83}, and so an isomorphism $H^0\big(X,\widetilde{\omega_R}(m)\big) \cong H^0\big(U,\widetilde{\omega_R}(m)\big)$. Likewise, we have $H^0\big(X,\mathcal{O}_{X}(m)\big) \cong H^0\big(U,\mathcal{O}_{X}(m)\big)$. So it follows from $(\ref{CoherentSheavesIdentity})$ and Lemma \ref{Helpful} that
$$
\omega_R \cong \bigoplus_{m \in \mathbb{Z}} H^0\big(X,\widetilde{\omega_R}(m)\big) \cong \bigoplus_{m \in \mathbb{Z}}H^0\big(X,\mathcal{O}_{X}(-d+a+m)\big) \cong R(-d+a),
$$
and $R$ is quasi-Gorenstein, as desired.
\end{proof}

\begin{remark}
One could try to prove results similar to Theorem \ref{GrGorenstein} for non standard graded rings. It is worth pointing out that examples of non Cohen-Macaulay quasi-Gorenstein, non standard graded rings constructed by using ample invertible sheaves are given in \cite{C14} and examples constructed by using non-integral $\mathbb{Q}$-divisors are given in Example \ref{DefNormal}(2), while examples that are standard graded are easily constructed as in Example \ref{DefNormal}(1).
\end{remark}

The following proposition shows ubiquity of quasi-Gorestein graded rings, which is an unpublished result due to K-i.Watanabe.

\begin{proposition}[K-i.Watanabe]
Let $X$ be an integral normal projective variety of dimension at least 2 defined over an algebraically closed field $k$. Then there exists a quasi-Gorenstein, Noetherian normal graded domain $R=\bigoplus_{n \ge 0} R_n$ with $R_0=k$ such that $X \simeq \Proj(R)$.
\end{proposition}

\begin{proof}
The proof cited in \cite[Proposition 5.9]{Sh17} applies directly to our case after dropping the assumption that $H^i(X,\mathcal{O}_X)=0$ for $0 < i < \dim X$.
\end{proof}

\section{Failure of deformation of quasi-Gorensteinness}

In view of Theorem \ref{Theorem1} $(7)$ together with \cite[Theorem 2.9]{TT16}, it seems to be promising that the quasi-Gorenstein property deforms (at least in equal-characteristic zero). However, counterexamples exist in both of prime characteristic and equal-characteristic zero cases.

\begin{counterexample}
\label{counterexample}
Suppose that $k$ is a field of either characteristic $2$ or of characteristic zero. Let us define $S$ to be the Segre product:
$$
S:=k[x,y,z]/(x^3)\  \# \ k[a,b,c]/(a^3),
$$
i.e. $S$ is the graded direct summand ring of the complete intersection ring $k[x,y,z,a,b,c]/(x^3,a^3)$ generated by the set of monomials $G:=\{xa,xb,xc,ya,yb,yc,za,zb,zc\}$. By \cite[Theorem (4.3.1)]{GW78}, $S$ is quasi-Gorenstein. By \cite[Proposition (4.2.2)]{GW78}, $S$ has dimension $3$ and it has depth $2$ by \cite[Proposition (4.1.5)]{GW78}. We define the homomorphism $\varphi:k[Z_1,\ldots,Z_9]\rightarrow S$ by setting $Z_i\mapsto G_i$. Then the ideal $\mathfrak{b}:=\text{ker}\ \varphi$ of $k[Z_1,\ldots,Z_9]$ is generated by the $2$-sized minors of the matrix $M:=\left(\begin{array}{ccc}
Z_{1} & Z_{2} & Z_{3}\\
Z_{4} & Z_{5} & Z_{6}\\
Z_{7} & Z_{8} & Z_{9}
\end{array}\right)$	
as well as the elements 
\begin{center}
\begin{align}\label{ExtraGenerators} &Z_1^3,Z_2^3,Z_3^3,Z_4^3,Z_7^3,\nonumber \\ & Z_1^2Z_2,Z_1^2Z_3,Z_1Z_2^2,Z_1Z_3^2,Z_2^2Z_3,Z_2Z_3^2,Z_1Z_2Z_3, \nonumber\\ & Z_1^2Z_4,Z_1^2Z_7,Z_1Z_4^2,Z_1Z_7^2,Z_4^2Z_7, Z_4Z_7^2,Z_1Z_4Z_7.
\end{align}
\end{center} 

So we have $S=k[Z_1,\ldots,Z_9]/\mathfrak{b}$. Now set $A:=k[Z_1,\ldots,Z_9,Y]$ and let $\mathfrak{a}$ be the ideal of $A$ generated by the equations (\ref{ExtraGenerators}) as well as the $2$-sized minors of the matrix $M$ with two exceptions: $Z_4Z_7Y-Z_6Z_8+Z_5Z_9$ instead of the determinant of $\left(\begin{array}{cc}
Z_{5} & Z_{6}\\
Z_{8} & Z_{9}
\end{array}\right)$ and $Z_1Z_7Y-Z_3Z_8+Z_2Z_9$ instead of the determinant of $\left(\begin{array}{cc}
Z_{2} & Z_{3}\\
Z_{8} & Z_{9}
\end{array}\right)$. 
Let us set $R:=A/\mathfrak{a}$ and suppose that $y$ is the image of $Y$ in $R$. Thus, we have $S=R/yR$. With the aid of the following Macaulay$2$ commands, one can verify that $y \in R$ is a regular element and $R$ is not quasi-Gorenstein.
	
$\\
i1 : A=\text{QQ}[Z_1..Z_9,Y,\text{Degrees}=>\{9:1,0\}]\  \\
o1 = A\\	
o1 : \text{PolynomialRing}\\
i2 : a=\text{ideal}(Z_6*Z_7-Z_4*Z_9,Z_5*Z_7-Z_4*Z_8,Z_3*Z_7-Z_1*Z_9,Z_2*Z_7-Z_1*Z_8,Z_3 
*Z_5-Z_2*Z_6,Z_3*Z_4-Z_1*Z_6,Z_2*Z_4-Z_1*Z_5,Z_4*Z_7*Y-Z_6*Z_8+Z_5*Z_9,Z_1*Z_7*Y 
-Z_3*Z_8+Z_2*Z_9,\ \ \ 
Z_1^3,Z 
_2^3,Z_3^3,Z_4^3,Z_7^3,\ \ \ 
Z_1^2*Z_2,Z_1^2*Z_3,Z_1*Z_2^2,Z_1*Z_3^2,Z_2^2*Z_3,Z_2*Z_3^2,Z_1*Z_2*Z_3,\ \ \ 
Z_1^2*Z_4,Z_1^2*Z_7,Z_1*Z_4^2,Z_1*Z_7^2,Z_4^2*Z_7,Z_4*Z_7^2,Z_1*Z_4*Z_7);	\\
o2 : \text{Ideal\  of}\  A\\
i3 : c=\text{ideal}(Z_1^3,Z_2^3,Z_3^3,Z_4^3,Z_7^3,Z_4*Z_7*Y-Z_6*Z_8+Z_5*Z_9);	\\
o3 : \text{Ideal\  of}\  A	\\
i4 : \text{codim}\  c == \text{codim}\  a	\\
o4 = \text{true}\\
i5 : \text{codim}\  c == 6	\\
o5 = \text{true}	\\
i6 : d=c:a;\\	
o6 : \text{Ideal\  of}\  A	\\
i7 : C=\text{module}(d)/\text{module}(c);	\\
i8 : N=C/((\text{ideal\  gens\  ring\ } C)*C);	\\
i9 : \text{numgens\  source\  basis}\ N	\\
o9 = 9	\\
i10 : a:Y == a	\\
o10 = \text{true}
$

Thus, the canonical module of $R$, which is the module $C$ in the above Macaulay2 code, is generated minimally by $9$ elements.  Note that the last command shows that $y$ is a regular element of $R$.  We remark that the quasi-Gorenstein local ring $S=R/yR$ is Gorenstein on its punctured spectrum, which also shows that the depth condition of Theorem \ref{Theorem1}(6) is necessary and  is sharp. Also we remark that,  replacing $QQ$ with $ZZ/\text{ideal}(2)$ in the first command of the above Macaulay2 code, leads to the same conclusion.
\end{counterexample}

Thus, we obtain the following result.

\begin{theorem}
\label{qGor}
There exists an example of a local Noetherian ring $(R,\fm)$, together with a regular element $y \in \fm$ such that the following property holds: $R/yR$ is quasi-Gorenstein and $R$ is not quasi-Gorenstein.
\end{theorem}

\begin{remark}
In spite of Counterexample \ref{counterexample}, the quasi-Gorenstein analogue of Ulrich's result \cite[Proposition 1]{U84} holds: A quasi-Gorenstein ring which is a homomorphic image of a regular ring and which is a complete intersection at codimension $\le 1$ has a deformation to an excellent unique factorization domain in view of \cite[Proposition 3.1]{Ta17}.
\end{remark}

The local ring $(R,\fm)$ constructed in Counterexample \ref{counterexample} is not normal. At the time of preparation of this paper, we do not have any concrete counterexample for the deformation of quasi-Gorensteinness in the context of normal domains. For standard graded normal domains, we have Theorem \ref{GrGorenstein}.

\section{Construction of quasi-Gorenstein rings which are not Cohen-Macaulay}

In this section, we offer three different potential instances of quasi-Gorenstein normal domains and we are curious to know whether or not any of these instances of quasi-Gorenstein normal (local) domains admits a deformation to a quasi-Gorenstein ring.

\begin{example}
\label{DefNormal}
\begin{enumerate}
\item
Let $k$ be any field with $\Char(k)\neq 3$ and suppose that $S$ is the Segre product of the cubic Fermat hypersurface: 
$$
k[x,y,z]/(x^3+y^3+z^3) \# k[a,b,c]/(a^3+b^3+c^3).
$$
Then in view of \cite{GW78}, $S$ is a quasi-Gorenstein normal domain of dimension $3$ and depth $2$ such that $\Proj(S)$ is the product of two elliptic curves and so $\Proj(S)$ is an Abelian surface. In contrast to Counterexample \ref{counterexample}, we expect that any deformation of $S$ would be again quasi-Gorenstein. In view of Theorem \ref{Theorem1}(1), perhaps it is worth remarking that, when  characteristic of $k$ varies over the prime numbers distinct from $3$, $S$ can be either $F$-pure or non-$F$-pure. In the case when $S$ is $F$-pure, any deformation of the local ring of the affine cone attached to $\Proj(S)$ is quasi-Gorenstein due to Theorem \ref{Theorem1}(1). On the other hand, if $\Char(k)=0$, then any standard homogeneous deformation of $S$ is quasi-Gorenstein in view of Theorem \ref{GrGorenstein}.

\item
In contrast to the previous example,  we hereby present an example of a non-Cohen-Macaulay quasi-Gorenstein normal graded domain $S$ with $\Proj(S)=\mathbb{P}_{k}^1\times \mathbb{P}_{k}^1$, where $k$ is a field either of  characteristic zero or  of  prime characteristic $p>0$ such that $p$ varies over a Zariski-dense open (cofinite) subset of prime numbers. The construction of such a quasi-Gorenstein normal domain is much more complicated than the previous one,   and the ring $S$ has to be a non-standard graded ring. Thanks to Demazure's theorem \cite{D88}, any (not necessarily quasi-Gorenstein) normal $\mathbb{N}_0$-graded ring $R=\bigoplus_{n\in\mathbb{N}_0}R_n$ with $X:=\Proj(R)=\mathbb{P}^1\times \mathbb{P}^1$ is the generalized section ring:
$$
R=R(X,D)=\bigoplus_{n\in\mathbb{N}_0}H^0\big(X,\mathcal{O}_X(\lfloor nD\rfloor)\big)
$$
for some rational coefficient Weil divisor $D\in \text{Div}(X,\mathbb{Q})=\text{Div}(X)\otimes \mathbb{Q}$ such that $nD$ is an ample Cartier divisor for some $n\gg 0$ (see \cite[Theorem, page 203]{W81} for the general statement of this fact and also \cite[\S1]{W81}  for the definitions and the background). We shall give an example of a rational Weil divisor $D$ on $\mathbb{P}^1\times \mathbb{P}^1$ whose generalized section  ring $R(X,D)$ is a non-Cohen-Macaulay quasi-Gorenstein normal domain with $a$-invariant $5$, and we will also present $R(X,D)$ explicitly as the Segre product of two hypersurfaces.\footnote{A non-Cohen-Macaulay section ring, whose projective scheme is $\mathbb{P}^1\times \mathbb{P}^1$, is given in \cite[Example (2.6)]{W81}. Here a non-Cohen-Macaulay quasi-Gorenstein normal domain will be explicitly given.} On the genus zero smooth curve $\mathbb{P}^1=\Proj(k[x,y])$ (respectively, with different coordinates, $\mathbb{P}^1=\Proj(k[w,z])$), consider the $\mathbb{Q}$-divisor 
$$
D_1:=2P_0-\sum_{i=1}^35/8P_i,
$$
where $P_i$ corresponds to the prime ideal, $x+iy$, for $i=0,\ldots,3$, respectively,
$$
D_2:=5Q_0-\sum_{i=1}^91/2Q_i,
$$
where $Q_i$ corresponds to the prime ideal $w+iz$ for $i=0,\ldots,9$. We follow the notation used in the end of the statement of  \cite[Theorem (2.8)]{W81}, and then we have $D_1'=\sum_{i=1}^37/8P_i$ and $D'_2=\sum_{i=1}^91/2Q_i$. Using the fact  that $K_{\mathbb{P}^1}=\mathcal{O}_{\mathbb{P}^1}(-2)$ is the canonical divisor of $\mathbb{P}^1$, one can easily verify that both of the $\mathbb{Q}$-divisors $K_{\mathbb{P}^1}+D'_1-5D_1$ and $K_{\mathbb{P}^1}+D'_2-5D_2$ are principal (integral) divisors and hence by \cite[Corollary (2.9)]{W81}, one can conclude that the section rings $G:=R(\mathbb{P}^1,D_1)$ and $G':=R(\mathbb{P}^1,D_2)$ are both Gorenstein rings with $a$-invariant $5$ (see also \cite[Example (2.5)(b)]{W81} and \cite[Remark (2.10)]{W81}). It follows that the Segre product $S:=G \# G'$ is a quasi-Gorenstein ring. In the sequel, we will give a presentation of $S$ and we show that it is not Cohen-Macaulay.

\begin{itemize}
\item
\textbf{Presentation of $G'$:}
Applying \cite[Chapter IV, Theorem 1.3 (Riemann-Roch)]{H83} we have $H^0\big(\mathbb{P}^1,\mathcal{O}_{\mathbb{P}^1}(\lfloor 2nD_2 \rfloor)\big)=H^0\big(\mathbb{P}^1,\mathcal{O}_{\mathbb{P}^1}(n)\big)$ is an $(n+1)$-dimensional vector space for each $n\ge 0$ (because $\lfloor 2nD_2 \rfloor$ has degree $n$, $K_{\mathbb{P}^1}-\lfloor 2nD_2 \rfloor=\mathcal{O}_{\mathbb{P}^1}(-n-2)$ is not generated by global sections and $\mathbb{P}^1$ has genus zero). More precisely, we have $\lfloor 2nD_2 \rfloor=10nQ_0-\sum_{i=1}^9nQ_i\sim\mathcal{O}_{\mathbb{P}^1}(n)$ which yields
\begin{align*}	 	
H^0\big(\mathbb{P}^1,\mathcal{O}_{\mathbb{P}^1}(\lfloor 2nD_2 \rfloor)\big)&=\Big\{f/g\in k(w,z)~\Big|~\text{div}(f/g)+10nQ_0-\sum_{i=1}^9nQ_i\ge 0\Big\}&\\&=\Big\{\Big(\big(\prod_{i=1}^{9}(w+iz)^{n}\big)f\Big)/w^{10n}~\Big|~f\in k[w,z]_{[n]}\Big\}.
\end{align*}
Consequently, $G'_{[2n]}$ is generated by $G'_{[2]}$ for each $n\ge 2$ (as the elements of the ring $G'$). Similarly, we can see that $H^0\big(\mathbb{P}^1,\mathcal{O}_{\mathbb{P}^1}(\lfloor 9D_2 \rfloor )\big)$ is the $1$-dimensional $k$-vector space spanned by $\big(\prod_{i=1}^{9}(w+iz)^{5}\big)/w^{45}$ which clearly provides us with a new generator of our section ring $G'$. One can then observe that, for $n\neq 4$,  $H^0\big(\mathbb{P}^1,\mathcal{O}_{\mathbb{P}^1}(\lfloor (2n+1)D_2 \rfloor)\big)$ is either zero for $n\le 3$ or it is an $(n-3)$-dimensional vector space  generated by ${G'}_{[9]}$ and ${G'}_{[2n-8]}$. It follows that $G'$ has three generators and since it has dimension $2$, we get
$$
G'=k[A',B',C']/(f)
$$
for some irreducible element $f \in k[A',B',C']$ of degree $18$, such that $A'$ and $B'$ have degree $2$ while $C'$ has degree $9$. Namely, $f=C'^2-(\prod_{i=1}^9(A'+iB'))$.

\item
\textbf{Presentation of $G$:} Similarly as in the previous part, for any $m\ge 0$ and $0\le k\le 7$, setting $0\neq n:=8m+k$, we can observe that  
\begin{align*}
&H^0\big(\mathbb{P}^1,\mathcal{O}_{\mathbb{P}^1}(\lfloor nD_1 \rfloor)\big)=&\\&\begin{cases}
(m+1)\text{-dimensional vector space\ }H^0\big(\mathbb{P}^1,\mathcal{O}_{\mathbb{P}^1}(m)\big),& k\overset{3}{\equiv} 0  \\
m\text{-dimensional vector space\ }H^0\big(\mathbb{P}^1,\mathcal{O}_{\mathbb{P}^1}(m-1)\big), & k\overset{3}{\equiv} 1 \\
\max\{0,(m-1)\}\text{-dimensional vector space\ }H^0\big(\mathbb{P}^1,\mathcal{O}_{\mathbb{P}^1}(m-2)\big),& k\overset{3}{\equiv} 2
\end{cases}
\end{align*}
that $G_{[n]}=H^0\big(\mathbb{P}^1,\mathcal{O}_{\mathbb{P}^1}(\lfloor nD_1 \rfloor)\big)$ is generated by $G_{[n-8]}$ and $G_{[8]}$ in the case where $m\ge 2$ and $n \ne 18,21$, that $G_{[6]}$, $G_{[9]}$, $G_{[12]}$ and $G_{[15]}$ are generated by $G_{[3]}$,  that $G_{[11]}$ (respectively, $G_{[14]}$) is generated by $G_{[8]}$ and $G_{[3]}$ (respectively, $G_{[11]}$ and $G_{[3]}$), that $G_{[18]}$ (respectively, $G_{[21]}$) is generated by $G_{[3]}$ and $G_{[15]}$ (respectively, $G_{[3]}$ and $G_{[18]}$) and that $G$ is zero in the remained unmentioned degrees. Consequently, $$G=k[A,B,C]/(g)$$ such that $\deg(A)=3$, $B$ and $C$ are of degree $8$ and $g=A^8-\prod_{i=1}^3(B+iC)$ (Note that $A$ corresponds to the element $\prod_{i=1}^3(x+iy)^2/x^6$, $B$ corresponds to $\prod_{i=1}^3\big((x+iy)^5x\big)/x^{16}$ and $C=\prod_{i=1}^3\big((x+iy)^5y\big)/x^{16}$).

\item
\textbf{Non-Cohen-Macaulayness of $S=G\#G'$:} Note that by Serre duality theorem,

\begin{align*}
H^{1}\big(\mathbb{P}^{1},\mathcal{O}_{\mathbb{P}^{1}}(\lfloor3D_{2}\rfloor)\big)=\text{Hom}\big(\mathcal{O}_{\mathbb{P}^{1}}(\lfloor3D_{2}\rfloor),K_{\mathcal{\mathbb{P}}^{1}}\big)&=H^{0}\Big(\mathbb{P}^{1},\mathscr{H}om\big(\mathcal{O}_{\mathbb{P}^{1}}(\lfloor3D_{2}\rfloor),\mathcal{O}_{\mathbb{P}^{1}}(-2)\big)\Big)
&\\ &=H^{0}\Big(\mathbb{P}^{1},\mathscr{H}om\big(\mathcal{O}_{\mathbb{P}^{1}}(-3),\mathcal{O}_{\mathbb{P}^{1}}(-2)\big)\Big)
&\\ &=H^{0}\big(\mathbb{P}^{1},\mathcal{O}_{\mathbb{P}^{1}}(1)\big)
\end{align*}

is a non-zero $2$-dimensional vector space. Thus,
\begin{align*}
H_{S^{+}}^{2}(S)_{[3]}&=H^{1}\big(\mathbb{P}^{1}\times\mathbb{P}^{1},\mathcal{O}_{\mathbb{P}^{1}}(\lfloor3D_{1}\rfloor)\boxtimes\mathcal{O}_{\mathbb{P}^{1}}(\lfloor3D_{2}\rfloor)\big)&\\&=\bigg(H^{0}\big(\mathbb{P}^{1},\mathcal{O}_{\mathbb{P}^{1}}(\lfloor3D_{1}\rfloor)\big)\otimes H^{1}\big(\mathbb{P}^{1},\mathcal{O}_{\mathbb{P}^{1}}(\lfloor3D_{2}\rfloor)\big)\bigg)&\\&\oplus \bigg(H^{1}\big(\mathbb{P}^{1},\mathcal{O}_{\mathbb{P}^{1}}(\lfloor3D_{1}\rfloor)\big)\otimes\underset{=G'_{[3]}=0}{\underbrace{H^{0}\big(\mathbb{P}^{1},\mathcal{O}_{\mathbb{P}^{1}}(\lfloor3D_{2}\rfloor)\big)}}\bigg)
&\\& = G_{[3]}\otimes H^{0}\big(\mathbb{P}^{1},\mathcal{O}_{\mathbb{P}^{1}}(1)\big)
&\\&\neq 0,
\end{align*}
which implies that $S$ is not-Cohen-Macaulay as required.
\end{itemize}
	 
\item
We give an explicit construction of a unique factorization domain (so, being quasi-Gorenstein in our case), not being Cohen-Macaulay of depth 2 with arbitrarily large dimension, as an invariant subring. Fix a prime number $p \ge 5$ and an algebraically closed field $k$ of characteristic $p$. Consider the $k$-automorphism on the polynomial algebra $k[x_1,\ldots,x_{p-1}]$ defined by
\begin{align*}
\sigma(x_1) &=x_1,\\
\sigma(x_2) &=x_2+x_1,\\
&\vdots                    \\
\sigma(x_{p-1}) &=x_{p-1}+x_{p-2}.
\end{align*}
Now we have $\sigma((x_1,\ldots,x_{p-2}))=(x_1,\ldots,x_{p-2})$ which is a prime ideal, so $\sigma$ gives rise to an action on the localization $R:=k[x_1,\ldots,x_{p-1}]_{(x_1,\ldots,x_{p-2})}$. Let $\fm$ be the unique maximal ideal of $R$. Let $\langle \sigma \rangle$ be the cyclic group generated by $\sigma$. Then the ring of invariants $R^{\langle \sigma \rangle}$ enjoys the following properties:

\begin{enumerate}
\item[$\bullet$]
$R^{\langle \sigma \rangle}$ is a local ring which is essentially of finite type over $k$, $R^{\langle \sigma \rangle}$ is a unique factorization domain with a non Cohen-Macaulay isolated singularity, $\dim R^{\langle \sigma \rangle}=p-2$ and $\depth R^{\langle \sigma \rangle}=2$. In particular, $R^{\langle \sigma \rangle}$ is quasi-Gorenstein.
\end{enumerate}

Since $R$ has characteristic $p$, $\sigma$ generates the $p$-cyclic action by construction. Then $R^{\langle \sigma \rangle} \hookrightarrow R$ is an integral extension and we thus have $\dim R^{\langle \sigma \rangle}=p-2$. Quite obviously,
$$
\big(\sigma(x_1)-x_1,
\sigma(x_2)-x_2,\ldots,\sigma(x_{p-1})-x_{p-1}\big)=\big(x_1,\ldots,x_{p-2}\big)
$$
is an $\fm$-primary ideal. By \cite[Lemma 3.2]{Pe83} (see also \cite{Fo81} for related results), the map $R^{\langle \sigma \rangle} \to R$ ramifies only at the maximal ideal. Since $R$ is regular, $R^{\langle \sigma \rangle}$ has only isolated singularity. By \cite[Corollary 1.6]{Pe83}, we have $\depth R^{\langle \sigma \rangle}=2$. Since $R^{\langle \sigma \rangle}$ has dimension $p-2 \ge 3$, we see that $R^{\langle \sigma \rangle}$ is not Cohen-Macaulay. It remains to show that $R^{\langle \sigma \rangle}$ is a unique factorization domain. For this, let us look at the action of $\langle \sigma \rangle$ on $k[x_1,\ldots,x_{p-1}]$. Then by \cite[Proposition 16.4]{F73}, $k[x_1,\ldots,x_{p-1}]^{\langle \sigma \rangle}$ is a unique factorization domain and we have
$$
R^{\langle \sigma \rangle}=\big(k[x_1,\ldots,x_{p-1}]_{(x_1,\ldots,x_{p-2})}\big)^{\langle \sigma \rangle}=\big(k[x_1,\ldots,x_{p-1}]^{\langle \sigma \rangle}\big)_{k[x_1,\ldots,x_{p-1}]^{\langle \sigma \rangle} \cap (x_1,\ldots,x_{p-2})}.
$$
Since being a unique factorization domain is preserved under localization, it follows that $R^{\langle \sigma \rangle}$ is a unique factorization domain, as desired. The paper \cite{MS11} examines more examples of non Cohen-Macaulay domains that are unique factorization domains.
\end{enumerate}
\end{example}

\begin{acknowledgement}
The authors are grateful to A. K. Singh for suggesting Example \ref{DefNormal}$(2)$ as well as his comments on a first draft of this paper. Our thanks are also due to L. E. Miller. The first author was partially supported by JSPS Grant-in-Aid for Scientific Research(C) 18K03257. The second author was partially supported by JSPS Grant-in-Aid for Young Scientists 20K14299 and JSPS Overseas Research Fellowships.
\end{acknowledgement}

\end{document}